\documentclass[11pt]{article}

\usepackage[all]{xy}
\usepackage{graphicx}
\usepackage{amsmath}
\usepackage{amsfonts}
\usepackage{amssymb}
\usepackage{amsthm}
\usepackage{amscd}
\usepackage{textcomp}
\usepackage{hyperref}

\newcommand{\cB}{\mathcal{B}}

\DeclareFontEncoding{OT2}{}{} 
  \newcommand{\textcyr}[1]{%
    {\fontencoding{OT2}\fontfamily{wncyr}\fontseries{m}\fontshape{n}%
     \selectfont #1}}
\newcommand{\Sha}{{\mbox{\textcyr{Sh}}}}

\newcommand{\cA}{\mathcal{A}}

\newcommand{\Kbar}{\overline{K}}

\newcommand{\cross}{\times}
\newcommand{\hra}{\hookrightarrow}
\newcommand{\ra}{\rightarrow}
\newcommand{\lra}{\longrightarrow}

\newcommand{\tensor}{\otimes}
\newcommand{\vphi}{\varphi}

\newcommand{\ssstyle}{\scriptscriptstyle}

\newcommand{\Q}{\mathbf{Q}}

\newcommand{\Z}{\mathbf{Z}}
\newcommand{\F}{\mathbf{F}}

\newcommand{\kbar}{\overline{k}}

\newcommand{\isom}{\cong}

\renewcommand{\O}{\mathcal{O}}

\DeclareMathOperator{\Gal}{Gal}
\DeclareMathOperator{\im}{im}

\DeclareMathOperator{\tor}{tor}
\DeclareMathOperator{\Spec}{Spec}

\DeclareMathOperator{\ur}{ur}

\DeclareMathOperator{\Hom}{Hom}

\theoremstyle{plain}
\newtheorem{theorem}{Theorem}[section]
\newtheorem{proposition}[theorem]{Proposition}

\newtheorem{lemma}[theorem]{Lemma}
\newtheorem{conjecture}[theorem]{Conjecture}

\theoremstyle{definition}

\theoremstyle{remark}

\numberwithin{equation}{section}

 1

\newcommand{\thetitle}
{Constructing non-trivial elements of the Shafarevich-Tate group of an Abelian Variety over a Number Field}

\begin{document}

\title{\thetitle}
\author{Amod Agashe\footnote{The first author was
supported by 
National 
Security Agency Grant No. Hg8230-10-1-0208.} \ \ and Saikat Biswas}
\maketitle

\begin{abstract}
Let $A$ and $B$ be abelian varieties over a number field $K$ such that 
$A[n]\isom B[n]$ over $K$ for some integer $n$. 
If $A$ has Mordell-Weil rank $0$, 
then we show that under certain additional hypothesis,
there is an injection $B(K)/nB(K) \hra \Sha(A/K)$, where $\Sha(A/K)$ is the Shafarevich-Tate group of $A$. This generalizes work of Cremona and Mazur. 
We also extend this general result to show that 
if the Mordell-Weil rank~$r_B$ of $B$ 
is bigger than the Mordell-Weil 
rank~$r_A$ of $A$ (which need not be zero now), 
then under certain extra hypotheses, 
$n^{r_B-r_A}$ divides the product of the order of $\Sha(A/K)$ and the Tamagawa numbers of $A$. When in addition $r_A = 0$, $n$ is prime, 
and $A$ and~$B$ are optimal elliptic curves of the same conductor,
this divisibility is
predicted by the second part of the Birch and Swinnerton-Dyer Conjecture, and thus provides new
theoretical evidence towards the conjecture.
\end{abstract}

\section{Introduction}\label{intro}

 Let $E/\Q$ be an optimal elliptic curve over $\Q$ having Mordell-Weil
rank $r$. Let $L_{E,\Q}(s)$ be the $L$-function associated to $E/\Q$ (see, e.g., \cite[C.\S 16]{silverman:AEC}). The Birch and Swinnerton-Dyer (BSD) Conjecture has two parts. The first part states
\begin{conjecture}[BSD I]\label{bsd1}
$L_{E,\Q}(s)$ has a zero at $s=1$ whose order is equal to $r$. 
\end{conjecture}
The order of vanishing at $s=1$ of $L_{E,\Q}(s)$ is called the \emph{analytic rank} of~$E$. 
Let $\Omega_E$ denote the \emph{real volume} of $E$,
computed using a N\'eron differential on~$E$. Suppose now that $E$ has analytic rank $0$ so that $L_{E,\Q}(1)\neq 0$. 
In this case, thanks to the work of Kolyvagin and Logachev, we know that 
the Shafarevich-Tate group~$\#\Sha(E/\Q)$ of~$E$ is finite. 
If $p$ is a prime, then let
$c_{\scriptscriptstyle{E},p}$ denote the Tamagawa number of~$E$
at~$p$ (for a definition, see Section~\ref{sec:neron}).
Then the second part of the BSD Conjecture says:
\begin{conjecture}[BSD II]\label{bsd2}
$$\frac{L_{E,\Q}(1)}{\Omega_E}\overset{?}{=}\frac{\#\Sha(E/\Q)\,\cdot\,\prod_{p}c_{\scriptscriptstyle{E},p}}{(\#E(\Q)_{\tor})^2}$$
\end{conjecture}
All the terms appearing in the formula of Conjecture \ref{bsd2} are computable except $\#\Sha(E/\Q)$. Thus, we can think of Conjecture \ref{bsd2} as giving a conjectural value for $\#\Sha(E/\Q)$ by the formula $$\#\Sha(E/\Q)\overset{?}{=}\frac{L_{E,\Q}(1)}{\Omega_E}\,\cdot\,\frac{(\#E(\Q)_{\tor})^2}{\prod_{p}c_{\scriptscriptstyle{E},p}}$$ 
When the BSD Conjecture predicts that $\#\Sha(E/\Q)$ is non-trivial, it is often desirable to verify the prediction by exhibiting actual, non-trivial elements in $\Sha(E/\Q)$. Conversely, one wishes to explicitly exhibit non-trivial elements in $\Sha(E/\Q)$ without assuming the BSD Conjecture and then verify it by comparing the order of the constructed elements with that of the conjectural order of $\Sha(E/\Q)$. 

In \cite{cremona-mazur}, Cremona and Mazur examined optimal elliptic curves $E/\Q$ of rank~$0$ for which the BSD Conjecture predicts that $\Sha(E/\Q)\isom{\Z}/{p\Z}\,\oplus\,{\Z}/{p\Z}$ for some odd prime $p$. They made the surprising discovery that in many such instances, there exists an optimal elliptic curve $F/\Q$ of the same level as~$E$ whose rank is $2$
and such that $E[p]\isom F[p]$ over~$\Q$. The authors interpreted the situation by claiming that the two independent generators of the Mordell-Weil group $F(\Q)$ \emph{explain} the two independent generators modulo $p$ of $\Sha(E/\Q)$. However, they did not provide adequate theoretical explanation for their claim. This lacuna in \cite{cremona-mazur} was eventually filled in \cite{cmapp} where the authors discussed the details of a criterion that could be used to show that in the examples considered in \cite{cremona-mazur}, the group $F(\Q)/pF(\Q)$ embeds itself in the group $\Sha(E/\Q)$, and in this
sense, $F(\Q)$ ``explains'' elements of~$\Sha(E/\Q)$
(see in particular Proposition A.7 in \cite{cmapp}). 
The goal of this article is to give
a generalization and extension of the discussion in~\cite{cmapp}.


Starting now, we drop the hypothesis that $E$ has analytic rank zero.
The following theorem is a consequence of Theorem \ref{maintheorem}
(in Section~\ref{sec:mainthm}), which is the main result of this paper.

\begin{theorem}\label{mt}
Let $E$ and $F$ be elliptic curves with semistable reduction over $\Q$, having ranks $r_{\ssstyle{E}}$ and $r_{\ssstyle{F}}$ respectively. Let $N$ be an integer divisible by the finitely many primes of bad reduction for both $E$ and $F$. Let $n$ be an odd integer such that
$$gcd\left(n,\;N \cdot\;\#E(\Q)_{\tor}\cdot\;\prod_{\ell}c_{\scriptscriptstyle{F},\ell}\right)=1$$
Suppose further that $F[n]\isom E[n]$ over~$\Q$. 
Then
\begin{description}
\item[(i)] 
If $r_{\ssstyle{F}}>r_{\ssstyle{E}}$, then $n^{(r_{\ssstyle{F}}-r_{\ssstyle{E}})}$ divides $\left(\#\Sha(E/\Q)\cdot\;\prod_{\ell}c_{\scriptscriptstyle{E},\ell}\right)$.
\item[(ii)] If $gcd\left(n,\;\prod_{\ell}c_{\scriptscriptstyle{E},\ell}\right)=1$, then there exists a map $$\vphi\,:\,F(\Q)/nF(\Q) \to \Sha(E/\Q)$$ whose kernel has order at most $n^{r_{\ssstyle{E}}}$. In particular, 
if $r_{\ssstyle{E}}=0$, then $\vphi$ is an injection,
and moreover, it becomes an isomorphism if we further assume that for all primes $p$ dividing $n$, 
$\Sha(F/\Q)$ has trivial $p$-torsion.
\end{description}
\end{theorem}

It can be verified that the hypotheses in part (ii) of Theorem \ref{mt} are satisfied by most of the pairs $(E,F)$ of elliptic curves listed in Table 1 of \cite{cremona-mazur} and consequently, non-trivial elements of $\Sha(E/\Q)$ are ``explained'' by $F(\Q)$ in these examples.
We now give one such example. 
Consider the optimal elliptic curve $E$=2834D. This curve appears in Table 1 of \cite{cremona-mazur}. We find from the elliptic curve data in \cite{cremona:algs} that $E$ has rank $0$, \mbox{$\#E(\Q)_{\tor}=1$}, $\prod_{l}c_{\scriptscriptstyle{E},l}=1$ and that $E$ has good reduction at $p=5$. We now find that $E[5]\isom F[5]$, where $F$=2834C is an optimal curve of rank $2$ and for which $\#F(\Q)_{\tor}=1$ and $\prod_{l}c_{\scriptscriptstyle{F},l}=24$. Also, $F$ has good reduction at $p=5$. Hence, the triple $(E,F,5)$ satisfies the hypothesis in part (ii) of Theorem \ref{mt} and we conclude that $F(\Q)/5F(\Q)\hra\Sha(E/\Q)[5]$. Since $F(\Q)/5F(\Q)\isom{\Z/5\Z}\oplus{\Z/5\Z}$, this implies that $\Sha(E/\Q)$ has order at least~$25$. Indeed, we find that this agrees with the BSD conjecture which predicts that $\#\Sha(E/\Q)=25$. 

The next example is taken from \cite{stein:tamagawa} and assumes the validity of BSD II. Consider the optimal elliptic curves $E$=114C1 and $F$=57A1. The data in \cite{cremona:algs} shows that $E$ has rank $0$, \mbox{$\#E(\Q)_{\tor}=4$} and $\Sha(E/\Q)$ has trivial conjectural order. On the other hand, we find that $F$ has rank $1$ and $\prod_{l}c_{\ssstyle{F},l}=2$. Furthermore, we have $E[5]\isom F[5]$ over $\Q$. Thus, the triple $(E,F,5)$ satisfies the hypothesis in part (i) of Theorem \ref{mt} and we conclude that $5^{2-1}=5$ divides $\prod_{l}c_{\ssstyle{E},l}$. This agrees with the available data, according to which $\prod_{l}c_{\ssstyle{E},l}=20$.

We now compare our Theorem~\ref{mt} above to two similar results in the literature. 
We have the following special case of Theorem 3.1 proved in \cite{agashe-stein:visibility}, which is also motivated by \cite{cmapp}.

\begin{theorem}[Agashe-Stein]\label{agst}
Let $E$ and $F$ be elliptic curves that are subvarieties of an abelian variety $J$ over $\Q$, and such that $E\cap F$ is finite. Let $N$ be an integer divisible by the primes of bad reduction for $F$. Let $n$ be an odd integer such that
$$gcd\left(n,\;N\cdot\;\#(J/F)(\Q)_{\tor}\cdot\;\#F(\Q)_{\tor}\cdot\;\prod_{\ell}c_{\scriptscriptstyle{E},\ell}\cdot\;\prod_{\ell}c_{\scriptscriptstyle{F},\ell}\right)=1$$
where $c_{\scriptscriptstyle{E},\ell}$ (resp. $c_{\scriptscriptstyle{F},\ell}$) is the Tamagawa number of $E$ (resp. $F$) at a prime $\ell$. Suppose further that $F[n]\subset E$ as subgroups of~$J$. Then there exists a natural map
$$\vphi\;:\;F(\Q)/nF(\Q) \to \Sha(E/\Q)$$
such that $\ker{\vphi}$ has rank at most $n^r$, where $r$ is the rank of~$E$. In particular, if $E(\Q)$ has rank zero, then $\vphi$ is an injection.
\end{theorem}

Comparing Theorem \ref{agst} with Theorem \ref{mt} above, we see that both $E$ and $F$ occurring in Theorem \ref{agst} are required to be contained in some ambient variety $J$, whereas no such conditions are imposed in Theorem \ref{mt}. One consequence of this is that one can check the hypothesis $E(\Q)\isom F(\Q)$ in Theorem \ref{mt} for optimal elliptic curves~$E$ and $F$ of different levels. Next, the hypothesis that $gcd\left(n,\;\#(J/F)(\Q)_{\tor}\cdot\;\#F(\Q)_{\tor}\right)=1$ in Theorem \ref{agst} is not required in Theorem \ref{mt} to define the map $\vphi$. On the other hand, the hypothesis $F[n]\subset E$ in Theorem \ref{agst} is weaker than the hypothesis $F[n]\isom E[n]$, which makes it easier to verify.

We also have the following theorem, which is 
Theorem 6.1 of \cite{dws} specialized to our context:

\begin{theorem}[Dummigan-Stein-Watkins]\label{dws}
Let $E$ and~$F$ be semistable optimal elliptic curves over $\Q$ 
of level~$N$ such that rank$(F)>$ rank$(E)=0$. Suppose $p$ is an odd prime such that $$p\nmid\left(N \cdot \,\prod_{\ell\,|\,N}c_{\scriptscriptstyle{F},\ell}\right), $$ and such that 
$E[p]$ and~$F[p]$ are irreducible and isomorphic over~$\Q$.
Then $p$ divides $\#\Sha(E/\Q)$.
\end{theorem}

Note that the condition that 
$E[p]$ and~$F[p]$ are irreducible is stronger than the requirement 
in our Theorem~\ref{mt} that $p$ does not divide $\#E(\Q)_{\tor}$.
Also, the theorem above says that a prime divides $\#\Sha(E/\Q)$
under certain hypotheses, whereas
our theorem shows that an integer (that is not necessarily prime)
divides $\#\Sha(E/\Q)$, 
under weaker hypotheses.

Finally, note that neither of Theorems  \ref{agst} or \ref{dws} have
anything analogous to part~(i) of our Theorem~\ref{mt}, which
shows that under certain hypotheses, an integer divides
the product $\left(\#\Sha(E/\Q)\cdot\;\prod_{\ell}c_{\scriptscriptstyle{E},\ell}\right)$, which appears in the numerator on the right side of the second
part of the BSD conjecture (Conjecture~\ref{bsd2}).\footnote{After we had proved Theorem~\ref{mt}, we came to know that William Stein had sketched 
a proof of a result similar to part~(i) of our Theorem~\ref{mt} in \cite{stein:tamagawa}. His techniques are different from ours, in that
he uses group cohomology while we use flat cohomology -- there
is an analogy between the approaches though.}

\begin{proposition} \label{prop1}
Assume the hypotheses of part~(i) of Theorem~\ref{mt}, and suppose
in addition that 
$n$ is prime,
$r_E = 0$, and 
$E$ and~$F$ are optimal elliptic curves of conductor~$N$,
Then $n$
divides the quantity
$\frac{L_{E,\Q}(1)}{\Omega_E}\,\cdot\,{(\#E(\Q)_{\tor})^2}$.
\end{proposition}
\begin{proof}
We may view $E$ and~$F$ as abelian subvarieties of~$J_0(N)$.
Then the hypothesis that $F[n]\isom E[n]$ implies
that the newforms corresponding to~$E$ and~$F$ are congruent
modulo~$n$.
Our result then follows by~\cite[Prop~1.5]{agashe:visfac}. 
\end{proof}
The second part of BSD conjecture  predicts
that the quantity
$\frac{L_{E,\Q}(1)}{\Omega_E}\,\cdot\,{(\#E(\Q)_{\tor})^2}$
is equal to 
the product $\left(\#\Sha(E/\Q)\cdot\;\prod_{\ell}c_{\scriptscriptstyle{E},\ell}\right)$. Under the conditions of Proposition~\ref{prop1},
part~(i) of our Theorem~\ref{mt} shows that 
$n$ does indeed divide the latter product, 
and thus provides new theoretical evidence
for the conjecture. 

While we have focussed on elliptic curves for simplicity so far,
we remark that the statement of Theorem \ref{mt} remains valid when the elliptic curves $E$ and $F$ are replaced by abelian varieties, 
and $\Q$ is replaced by a number field~$K$, provided that for every prime~$p$ dividing~$n$, the largest ramification index of any prime ideal of~$K$ lying over~$p$ is less than $p-1$ (see Theorem~\ref{maintheorem}). 
In fact, in the rest of this article, we will work directly with 
abelian varieties over number fields. 

The proof of our main theorem (Theorem~\ref{maintheorem}) uses
the interpretation of 
the Shafarevich-Tate group of 
an abelian variety in terms of
the flat cohomology groups of its N{\'e}ron model; we discuss this
interpretation (following Mazur~\cite{mazur:towers})
in Section~\ref{sec:neron}.
In Section~\ref{sec:kummer},
we prove some results that will be used in the proof
of the main theorem (Theorem~\ref{maintheorem}), which
is stated and proved in Section~\ref{sec:mainthm}.
\vspace{0.1 in}

\noindent {\it Acknowledgements:} 
The first author is grateful 
to B.~Mazur and B.~Conrad for answering some questions
that came up when writing this article. The second author
is grateful to E.~Aldrovandi and Mark van Hoeij for many
helpful clarifications as well as feedback. Thanks also to
James Borger, Kevin Buzzard, Pete Clark, Matthew Emerton and
Keerthi Sampath for their careful explanations of several key
concepts.

\section{ Flat Cohomology of N{\'e}ron Models}\label{sec:neron}

In this section, following Mazur~\cite{mazur:towers}, 
we express the 
Shafarevich-Tate group of 
an abelian variety in terms of
the flat cohomology groups of its N{\'e}ron model.

Let $K$ be a number field with ring of integers $\O_K$ and let $A/K$ be an abelian variety over $K$. Let $\cA$ 
denote the \emph{N{\'e}ron model} of $A/K$ over $X=\Spec{\O_K}$. Thus $\cA$ is separated and of finite type over $X$ with generic fiber $A$, and satisfies the \emph{N{\'e}ron mapping property}: for each smooth $X$-scheme $S$ with generic fiber $S_{K}$, the restriction map ${\Hom_{X}(S,{\cA})\ra\Hom_{K}(S_{K},A)}$ is bijective. 
There is an alternative definition of N{\'e}ron models that will be useful for our purpose. Recall that (\cite[\S II.1]{milne:etale}) an abelian variety $A$ defines a sheaf (which we also denote as $A$) for the smooth ({\'e}tale) topology over $\Spec{K}$. Consider the direct-image sheaf $j_{*}(A)$ for the smooth (or {\it{fpqf}} or {\'e}tale) topology over $X=\Spec{\O_K}$, where $j:\Spec{K}\hra{X}$ is the inclusion of the generic point. When the sheaf $j_{*}(A)$ is \emph{representable}, the smooth scheme $\cA\to X$ representing $j_{*}(A)$ will be called a N{\'e}ron model of $A$. Abusing notation, we identify the functor $j_{*}(A)$ itself as the N{\'e}ron model of $A$ and write $\cA=j_{*}(A)$. We then find that the \emph{N{\'e}ron mapping property} is equivalent to the isomorphism $\cA\isom{j_{*}j^{*}\cA}$ which follows directly from the fact that the functors $j_{*}$ and $j^{*}$ are \emph{adjoint} to each other. For each closed point (i.e., prime) $v$ in $X$, the fiber  $\cA_{v}={\cA}\cross_{X}{\Spec{k_v}}$ is a smooth commutative group scheme over $k_v$, where $k_v$ is the residue field of the local ring ${\O}_{X,v}$ at $v$ or equivalently, the residue field of the ring of integers ${\O}_{K,v}$ of $K_v$, the completion of $K$ at $v$. 
Let $\cA_{v}^{0}\subset{\cA_{v}}$ be the connected component of $\cA_v$ that contains the identity (also called the \emph{identity component}). Let $Z_{v}$ be the complement of $\cA_{v}^{0}$ in $\cA_{v}$. We find that $Z_{v}$ is nonempty for only a finite number of closed points $v$ in $X$ since $A$ has good reduction at almost all primes. Thus $Z=\bigcup_{v}Z_{v}$ is a closed subscheme of $\cA$ and we let $\cA^{0}\subset{\cA}$ be its open complement, which is easily seen to be a connected open smooth subgroup scheme of $\cA$. Consider the short exact sequence of sheaves for the flat topology (since both $\cA^{0}$ and $\cA$ are smooth group schemes) over $X$: 
\begin{equation}\label{SES}
0\to \cA^{0} \to {\cA} \to {\Phi}_A \to {0}
\end{equation}
where ${\Phi}_A$ (the quotient of $\cA$ by $\cA^{0}$) is considered as a \emph{skyscraper sheaf}: it is zero outside of the finite set of points $v$ in $X$ where $\cA_{v}$ is disconnected i.e. those primes $v$ at which $A$ has \emph{bad reduction}. If we regard $\Phi_A$ as an {\'e}tale sheaf over $X$ and denote by $\Phi_{A,v}$ its stalk at a prime $v$, then $\Phi_{A,v}$ can be considered as a finite, {\'e}tale group scheme over $\Spec{k_v}$. Equivalently, $\Phi_{A,v}$ is a finite abelian group equipped with an action of $\Gal({\kbar_v}/{k_v})$. Over $\Spec{k_v}$, we thus have an exact sequence of group schemes 

\begin{equation}\nonumber
0 \to \cA_v^0 \to \cA_v \to \Phi_{A,v} \to 0
\end{equation}


The group scheme $\Phi_{A,v}={\cA_{v}}/{\cA_{v}^{0}}$ of connected components is called the \emph{component group} of $\cA$ at $v$ and $c_{A,v}=\#{\Phi_{A,v}(k_v)}$ is called the \emph{Tamagawa number} of $A$ at $v$. Now consider the natural closed immersion $i_{v}:\Spec{k_v}\hra{X}$. We have that 

\begin{equation}\label{compdef}
\Phi_A=\bigoplus_{v}(i_{v})_{*}{{\Phi}_{A,v}}
\end{equation}
where the direct sum is taken over all $v$ or equivalently, over the finite set of $v$ where $A$ has bad reduction.

 
The short exact sequence (\ref{SES}) of {\'e}tale (or, flat) sheaves over $X$ induces a long exact sequence of flat (or {\'e}tale) cohomology groups:
\begin{equation}\label{LES}
0 \to \cA^{0}(X) \to \cA(X) \to \Phi_A(X)
\to H^1(X,\cA^0) \to H^1(X,\cA) \to H^1(X,\Phi_A)
\end{equation}
where we can write, for all $i$,
\begin{equation}\label{component}
H^i(X,\Phi_A)=\bigoplus_{v}H^i(\Spec{k_v},\Phi_{A,v})
\end{equation}
which follows from (\ref{compdef}). Now the N{\'e}ron mapping property implies that $\cA(X)\isom A(K).$ Now let $K_{v}$ be the completion of $K$ with respect to the valuation defined by the prime $v$. We denote this valuation also as $v$. Consider, for a given $v$, the map $$w_{v}:\;\;H^1(\Spec{K},A_{K}) \to H^1(\Spec{K_{v}},A_{K_{v}})$$ which corresponds to the usual restriction map for Galois cohomology:
$${\textrm{Res}}_{v}:\;\;H^1(K,A) \lra H^1(K_{v},A)$$ 
Let us define two subgroups of $H^1(\Spec{K},A_{/K})$ (or, of $H^1(K,A)$):
\begin{align*}
\Sigma &= \bigcap{\ker(w_{v})}=\bigcap{\ker({\textrm{Res}}_{v})},\;\;\;{\textrm{nonarchimedean}}\;v\\
\Sha &= \bigcap{\ker(w_{v})}=\bigcap{\ker({\textrm{Res}}_{v})},\;\;\;{\textrm{all}}\;v
\end{align*}
where $\Sha=\Sha(A/K)$ is the \emph{Shafarevich-Tate group} of $A$. These two subgroups fit into an exact sequence of the form
\begin{equation}\label{shasigma}
0 \lra \Sha \lra \Sigma \lra \bigoplus_{{\textrm{real}}\,v}H^1(K_{v},\,A({\Kbar}_{v}))
\end{equation}
where we sum the Galois cohomology groups on the right over all real archimedean valuations $v$ or equivalently, over the set of those real $v$ for which $A({\Kbar}_{v})$ is disconnected. Thus if $A({\Kbar}_{v})$ is connected for all real valuations $v$ of $K$, then $\Sigma\isom\Sha$. The sequence (\ref{shasigma}) may now be written as 
\begin{equation}\nonumber
0 \lra \Sha \lra \Sigma \lra {\Sigma}/{\Sha} \lra 0
\end{equation}
where ${\Sigma}/{\Sha}$ is a finite group of exponent two. In particular, for an odd prime $p$, we have $$\Sigma[p^{\infty}]\isom\Sha[p^{\infty}]$$ i.e. the $p$-primary component of $\Sigma$ and $\Sha$ are equal. Mazur has shown (see the Appendix to \cite{mazur:towers}) that 
\begin{equation} \nonumber 
\Sigma\isom{\textrm{Im}}\left[H^1(X,\cA^{0})\lra H^1(X,\cA)\right]
\end{equation} 
Thus, the exact sequence (\ref{LES}) becomes
\begin{equation}\label{mainseq}
0 \to \cA^0(X) \to A(K) \to \bigoplus_{v}\Phi_{A,v}(k_v) \to H^1(X,\cA^0) \to \Sigma \to 0
\end{equation}
In particular, staying away from $2$-primary components, we note that $\Sha$ may be expressed in two ways. First, the sequence 
\begin{equation}\nonumber 
\bigoplus_{v}\Phi_{A,v}(k_v) \to H^1(X,\cA^0) \to \Sha \to 0
\end{equation}
identifies $\Sha$ as a cokernel. Secondly, the sequence 
\begin{equation}\label{kernelsha}
0 \to \Sha \to H^1(X,\cA) \to \bigoplus_{v}H^1(\Spec{k_v},\Phi_{A,v}) 
\end{equation}
identifies $\Sha$ as a kernel.

\section{Kummer Theory of N{\'e}ron models}
\label{sec:kummer}

In this section, we prove some results that will be used in the proof
of the main theorem (Theorem~\ref{maintheorem}) in Section~\ref{sec:mainthm}.
The reader who is primarily interested in the proof of the main theorem
may safely skip the proofs of the results in this section on a first reading.

Multiplication by $n$ on the N{\'e}ron model $\cA$ produces a Kummer sequence involving flat cohomology groups analogous to the usual Kummer sequence involving Galois cohomology groups. This is explained by the next lemma. We let $\cA'$ be the inverse image of $n\Phi_A\subset\Phi_A$ under the natural surjection $\cA\to\Phi_A$ as in (\ref{SES}).

\begin{lemma}\label{kummerlemma}
Let $A$ be a semi-abelian variety over a number field $K$. Let $n$ be an odd integer and let $\cA[n]$ be the $n$-torsion of the N{\'e}ron model $\cA$ of $A$ over $X=\Spec{\O}_K$. Then there is an exact sequence
$$0 \to \cA'(X)/n\cA(X) \to H^1(X,\cA[n]) \to H^1(X,\cA)[n] \to 0.$$
\end{lemma}

\begin{proof}
By definition of $\cA'$, we have an exact sequence 
\begin{equation}\label{ses}
0 \to \cA' \to \cA \to {\Phi_A}/{n\Phi_A} \to 0
\end{equation}
of flat (or {\'e}tale) group schemes over $X$. In particular, we have an isomorphism $$\cA/\cA'\isom\Phi_A/n\Phi_A.$$ Now consider the multiplication-by-$n$ map on the exact sequence (\ref{SES}). We have a commutative diagram
$$\xymatrix{
& 0\ar[d] & 0\ar[d] & 0\ar[d] & \\
& \cA^0[n]\ar[d] & \cA[n]\ar[d] & \Phi_A[n]\ar[d] & \\
0\ar[r] & \cA^0\ar[r]\ar[d]^{n} & \cA\ar[r]\ar[d]^{n} & \Phi_A\ar[r]\ar[d]^{n} & 0\\
0\ar[r] & \cA^0\ar[r]\ar[d] & \cA\ar[r]\ar[d] & \Phi_A\ar[r]\ar[d] & 0\\
& \cA^0/n\cA^0\ar[d] & \cA/n\cA\ar[d] & \Phi_A/n\Phi_A\ar[d] & \\
& 0 & 0 & 0 &
}$$
and the snake lemma gives an exact sequence $$0 \to \cA^0[n] \to \cA[n] \to \Phi_A[n] \to \cA^0/n\cA^0 \to \cA/n\cA \to \Phi_A/n\Phi_A \to 0$$ Now $A$ is semi-abelian. Since multiplication by $n$ is surjective on tori as well as abelian varieties, it follows that it is also surjective on $\cA_v^0$ for all $v$. Consequently, the multiplication-by-$n$ map is surjective on $\cA^0$ so that $\cA^0/n\cA^0$ is trivial. The exact sequence above then implies that $\cA/n\cA\isom\Phi_A/n\Phi_A\isom\cA/\cA'$ so that we have $n\cA\isom\cA'$. It follows that we also have an exact sequence
\begin{equation}\label{ses:flat}
0 \to \cA[n] \to \cA \overset{n}{\to} \cA' \to 0
\end{equation}
of flat group schemes over $X$. The corresponding long exact sequence of flat cohomology groups yields the exact sequence $$0 \to \cA'(X)/n\cA(X) \to H^1(X,\cA[n]) \to H^1(X,\cA)[n] \to 0.$$ 
\end{proof}

Since the multiplication-by-$n$ map $\cA^0 \overset{n}{\to} \cA^0$ on $\cA^0$ is surjective, we have an exact sequence $$0 \to \cA^0[n] \to \cA^0 \to \cA^0 \to 0$$ of flat group schemes over $X$. The corresponding long exact sequence of flat cohomology groups yields the exact sequence $$0 \to \cA^0(X)/n\cA^0(X) \to H^1(X,\cA^0[n]) \to H^1(X,\cA^0)[n] \to 0$$
The Kummer sequence given above can be suitably modified under certain constraints on the Tamagawa numbers of $A$, as in the next proposition.

\begin{proposition}\label{main:prop}
Let $A$ be a semi-abelian variety over a number field $K$. Let $n$ be an odd integer such that
$$gcd\left(n,\;\prod_{v}c_{\scriptscriptstyle{A},v}\right)=1$$ Let $\widetilde{\cA[n]}\subset\cA[n]$ be any open quasi-finite subgroup scheme of $\cA[n]$ containing $\cA[n]^0:=\cA^0\cap\cA[n]$. Then there is an exact sequence
$$0 \to A(K)/nA(K) \to H^1(X,\widetilde{\cA[n]}) \to \Sha(A/K)[n] \to 0.$$
\end{proposition}

The statement and proof of Proposition \ref{main:prop} has been adapted from the discussion preceding Corollaries A.4 and A.5 in \cite{cmapp}. In particular, the statement of the proposition is almost the same as part (iii) of Corollary A.5 in \cite{cmapp}. Before we prove the proposition above, we
shall need the following lemma.

\begin{lemma}\label{lem}
Under the hypotheses of the above proposition, there are isomorphisms
$$A(K)\tensor{\Q_p/\Z_p} \isom H^0(X,\cA)\tensor{\Q_p/\Z_p}\isom H^0(X,\cA^0)\tensor{\Q_p/\Z_p}$$
$$\Sha(A/K)[p^{\infty}] \isom H^1(X,\cA^0)[p^{\infty}] \isom H^1(X,\cA)[p^{\infty}]$$ for any prime $p$ dividing $n$.
\end{lemma}

\begin{proof}
Consider the exact sequence (\ref{mainseq}). The order of the group $\Phi_A(X)\isom\bigoplus_{v}\Phi_{A,v}(\F_v)$ is the product $\prod_{v}{c_{\scriptscriptstyle{A},v}}$ of the Tamagawa numbers of $A$. Thus, the hypothesis of the proposition suggests that the finite group $\Phi_A(X)$ has trivial $p$-primary components for all primes $p$ dividing $n$. Now, in the sequence (\ref{mainseq}), we find that the injective map $\cA^0(X)\to A(K)$ has a cokernel which is finite and whose $p$-primary components are trivial (being a subgroup of $\Phi_A(X)$). This immediately implies, by invoking the snake lemma, that 
$$H^0(X,\cA^0)\tensor{\Q_p/\Z_p}\isom A(K)\tensor{\Q_p/\Z_p}\isom H^0(X,\cA)\tensor{\Q_p/\Z_p}$$ where the second isomorphism follows from the universal property of $\cA$. Similarly, the surjective map $H^1(X,\cA^0)\to\Sha(A/K)$ has a kernel which is finite and has trivial $p$-primary components (again, being a subgroup of $\Phi_A(X)$) so that, once again, we have $$\Sha(A/K)[p^{\infty}] \isom H^1(X,\cA^0)[p^{\infty}] \isom H^1(X,\cA)[p^{\infty}]$$
\end{proof}

\begin{proof}[Proof of Proposition \ref{main:prop}]
The sequences (\ref{ses}) and (\ref{ses:flat}) fit into a commutative diagram 
$$\xymatrix{
0\ar[r] & \cA[n]\ar[r] & \cA\ar[d]\ar[dr]^n\ar[r] & \cA'\ar[d]\ar[r] & 0 & \\
& 0\ar[r] & \cA'\ar[r] & \cA\ar[r] & {\Phi_A}/{n\Phi_A}\ar[r] & 0}$$
which we will call the \emph{Kummer diagram} for $A$ (\cite{mazur:towers}). Now any factor $p$ of $n$ does not divide the Tamagawa numbers $c_{\scriptscriptstyle{A},v}=\#\Phi_{A,v}(k_v)$ of $A$ at $v$, for all $v$. We note that, regarding $\Phi_{A,v}$ as a finite Galois module with a continuous action of $\Gal({\kbar}_v/{k_v})=\Gal(K_v^{\ur}/K_v)$, we have $$c_{\scriptscriptstyle{A},v}=\#\Phi_{A,v}(k_v)=\#H^0(K_v^{\ur}/K_v,\Phi_{A,v})=\#H^1(K_v^{\ur}/K_v,\Phi_{A,v})$$
where the last equality follows from the fact that the Herbrand quotient (\cite[\S VIII.4]{serre:localfields}) of the finite module $\Phi_{A,v}$ (in fact, any finite module) is $1$. The Galois cohomology group ~$H^i(K_v^{\ur}/K_v,\Phi_{A,v})$ is equal to the flat cohomology group $H^i(\Spec{k_v},\Phi_{A,v})$ for all $i$, where we regard $\Phi_{A,v}$ as a sheaf for the flat topology over $\Spec{k_v}$. Thus, we have $$c_{\scriptscriptstyle{A},v}=\#\Phi_{A,v}(k_v)=\#H^0(\Spec{k_v},\Phi_{A,v})=\#H^1(\Spec{k_v},\Phi_{A,v})$$ So the hypothesis on $p$ implies that the sheaf $\Phi_{A,v}$ is \emph{$p$-cohomologically trivial} over $\Spec{k_v}$ i.e the $p$-primary components of $H^i(\Spec{k_v},\Phi_{A,v})$ are zero for all $i=0,1$. By considering (\ref{component}), this implies that the sheaf $\Phi_A$ is \emph{$p$-cohomologically trivial} over $X$ for all primes $p$ dividing $n$ or equivalently, $\Phi_A/n\Phi_A$ is cohomologically trivial, i.e., for all $i$, $H^i(X,\Phi_A/n\Phi_A)=0$. Obviously, any subquotient $\Psi$ of $\Phi_A$ is $p$-cohomologically trivial as well. Considering the long exact sequence of flat cohomology groups corresponding to the bottom exact row of the Kummer diagram, this implies an isomorphism $$H^i(X,\cA')\isom H^i(X,\cA).$$
Applying this isomorphism to the exact sequence in Lemma~\ref{kummerlemma}, we get an exact sequence $$0 \to \cA(X)/n\cA(X) \to H^1(X,\cA[n]) \to H^1(X,\cA)[n] \to 0$$ which can also be written as
\begin{equation}\label{flat:kummer}
0 \to A(K)/nA(K) \to H^1(X,\cA[n]) \to \Sha(A/K)[n] \to 0
\end{equation}
by Lemma \ref{lem}. Now let $\cA[n]^0=\cA^0\cap\cA[n]$ be the largest open subgroup scheme of $\cA[n]$ every fiber of which is connected. Let $\widetilde{\cA[n]}\subset\cA[n]$ be any open subgroup scheme of $\cA[n]$ containing $\cA[n]^0$, i.e., $$\cA[n]^0\subset\widetilde{\cA[n]}\subset\cA[n]$$ We now have an exact sequence of group schemes over the flat topology of $X$:
\begin{equation} \nonumber 
0 \to \widetilde{\cA[n]} \to \cA[n] \to \Psi_A \to 0
\end{equation}
where $\Psi_A$ is a subquotient of $\Phi_A$. The hypothesis on $n$ implies that, for all $i=0,1$, $H^i(X,\Psi)$ has trivial $p$-primary components for all primes $p$ dividing $n$. In particular, we find that 
\begin{equation} \nonumber  
H^i(X,\widetilde{\cA[n]})=H^i(X,\widetilde{\cA[n]})\tensor \Z/n\Z \isom H^i(X,\cA[n])\tensor \Z/n\Z=H^i(X,\cA[n])
\end{equation}
The exact sequence (\ref{flat:kummer}) can then be written as
\begin{equation} \nonumber 
0 \to A(K)/nA(K) \to H^1(X,\widetilde{\cA[n]}) \to \Sha(A/K)[n] \to 0,
\end{equation}
as was to be shown.
\end{proof}

\section{Constructing Elements of $\Sha(A/\Q)$ using Mordell-Weil Groups}
\label{sec:mainthm}

The following theorem is the main theorem of this article.
It is an adaptation and generalization of the discussion in \cite{cmapp}. 
It can sometimes be used to construct non-trivial elements of $\Sha(A/K)$ 
and thus to provide nontrivial lower bounds for the order of $\Sha(A/K)$.

\begin{theorem}\label{maintheorem}
Let $A$ and $B$ be abelian varieties over a number field $K$, of ranks $r_{\ssstyle{A}}$ and $r_{\ssstyle{B}}$ respectively, such that $B$ has semistable reduction over $K$. Let $N$ be an integer divisible by the residue characteristics of the primes of bad reduction for both $A$ and $B$. Let $n$ be an odd integer such that for each prime $p\;|\;n$, we have $e_p<p-1$, where $e_p$ is the largest ramification index of any prime of $K$ lying over $p$, and such that
$$gcd\left(n,\;N\cdot\;\#A(K)_{\tor}\cdot\;\prod_{v}c_{\scriptscriptstyle{B},v}\right)=1$$
Suppose further that $B[n]\isom A[n]$ over $K$. Then there is a map
$$\vphi\,:\,B(K)/nB(K) \to H^1(K,A)$$ 
such that $\ker(\vphi)$ has order at most $n^{r_{\ssstyle{A}}}$. Moreover, we have that
\begin{description}
\item[(i)] if $r_{\ssstyle{B}}>r_{\ssstyle{A}}$, then $n^{r_{\ssstyle{B}}-r_{\ssstyle{A}}}$ divides the product $\left(\#\Sha(A/K)\cdot\;\prod_{v}c_{\ssstyle{A},v}\right)$.
\item[(ii)] if $gcd\left(n,\;\prod_{v}c_{\ssstyle{A},v}\right)=1$, then $\im(\vphi)\subset\Sha(A/K)$. In particular, if we also have that $r_{\ssstyle{A}}=0$, then there is an injection $$B(K)/nB(K)\hra\Sha(A/K)\ .$$
The injection above becomes an isomorphism if we further assume that for all primes $p$ dividing $n$, $\Sha(B/K)$ has trivial $p$-torsion.
\end{description}
\end{theorem}

For a discussion of what happens when $n$ is divisible by a prime
of bad reduction for either $A$ or $B$ or both, see the proof of Sublemma A.8 in \cite{cmapp} as well as the discussion immediately following it which completes the proof of Proposition A.7 therein.

In the rest of this section, we give the proof of Theorem~\ref{maintheorem}. 
To define the map $\vphi:B(K)/nB(K)\to H^1(K,A)$, note that the isomorphism ~$B[n]\isom A[n]$ over $K$ induces an isomorphism $H^1(K,B[n])\isom H^1(K,A[n])$ of Galois cohomology groups. Consequently, the Kummer sequence for both $A$ and $B$ fit in the following way
$$\xymatrix{
0\ar[r] & B(K)/nB(K)\ar[r] & H^1(K,B[n])\ar[r]\ar[d]^{\cong} & H^1(K,B)[n]\ar[r] & 0\\
0\ar[r] & A(K)/nA(K)\ar[r] & H^1(K,A[n])\ar[r] & H^1(K,A)[n]\ar[r] & 0}$$
The composition $B(K)/nB(K)\hra H^1(K,B[n])\isom H^1(K,A[n])\to H^1(K,A)[n]$ defines $\vphi$. The kernel of $\vphi$ is seen to be contained in the kernel of the map $H^1(K,A[n])\to H^1(K,A)[n]$ which is equal to the image of the injection $A(K)/nA(K)\hra H^1(K,B[n])$ by the exactness of the Kummer sequence for $A$. Since $n$ is coprime to $A(K)_{\tor}$, we find that $$\#A(K)/nA(K)=n^{r_{\ssstyle{A}}}$$ Thus $\ker(\vphi)$ has order at most $n^{r_{\ssstyle{A}}}$.

We now prove the rest of the statements in Theorem \ref{maintheorem}. Clearly, it suffices to prove the theorem when $n=p^m$ is the power of an odd prime~$p$, which is what we assume from now on. We begin by proving the following lemma.

\begin{lemma}
Under the hypotheses of the above theorem, there is an isomorphism $$\cA[p^m]\isom\cB[p^m]$$ over $X=\Spec{\O_K}$.
\end{lemma}

\begin{proof}
Let $U$ be an open subset of $X$ whose complement is the finite set of primes in $X$ having residue characteristic equal to $p$. 
Since the residue characteristics of the primes in $U$ are not equal to $p$, we find that the multiplication-by-$p^m$ map $$\cA_{/U}\overset{p^m}{\longrightarrow}\cA_{/U}$$ is \emph{{\'e}tale} (\cite[\S 7.3, Lemma 2(b)]{neronmodels}). Now consider the diagram
$$\xymatrix{
\cA[p^m]\ar[r]\ar[d] & U\ar[d]\\
\cA_{/U}\ar[r]^{p^m} & \cA_{/U}}$$
defining the kernel $\cA[p^m]_{/U}$. We find that $\cA[p^m]_{/U}$, being the fiber of multiplication-by-$p^m$ over the unit section $U\to\cA_{/U}$, is {\'e}tale (smooth) over $U$ with generic fiber $A[p^m]$ over $K$. Similar arguments hold for $B$ as well. 

The composition $A[p^m]\isom B[p^m]\hra B$ defines a map $A[p^m]\to B$ over $K$. By the N{\'e}ron mapping property over $U$, this map extends to a map $\cA[p^m]_{/U}\to\cB_{/U}$ over $U$, which factors through $\cB[p^m]$. Thus, we get a map $\cA[p^m]_{/U}\to\cB[p^m]_{/U}$. In a similar way, we get a map $\cB[p^m]_{/U}\to\cA[p^m]_{/U}$. The composition of these two maps over $U$ (in any order) induces the identity map over the generic fibers and hence, by the uniqueness of the N{\'e}ron mapping property, we get an isomorphism $$\cA[p^m]_{/U}\isom\cB[p^m]_{/U}$$ over $U$ which extends the isomorphism $A[p^m]\isom B[p^m]$ over $K$. Now $p$ is a prime of \emph{good} reduction for both $A$ and $B$, since $gcd(p^m,N)=1$. Therefore, over $X_v=\Spec{\O_{K,v}}$ with char$(v)=p$, we find that both $\cA[p^m]_{/{X_v}}$ and $\cB[p^m]_{/{X_v}}$ are finite flat group schemes of odd order with isomorphic generic fibers $A[p^m]_{/K_v}\isom B[p^m]_{/K_v}$. By Corollary 3.3.6 in \cite{raynaud-ppp}, we have an isomorphism $$\cA[p^m]_{/{X_v}}\isom\cB[p^m]_{/{X_v}}$$ over $X_v$. We can now `patch' the isomorphisms over all such $X_v$'s with the one over $U$ to get the desired isomorphism $$\cA[p^m]_{/X}\isom\cB[p^m]_{/X}$$ over $X$. 
\end{proof}

The isomorphism $\cA[p^m]_{/X}\isom\cB[p^m]_{/X}$ gives an isomorphism $$H^1(X,\cA[p^m])\isom H^1(X,\cB[p^m])$$ of flat cohomology groups. Consider the commutative diagram
\begin{equation}\label{diagram}
\xymatrix{
0\ar[r] & B(K)/p^mB(K)\ar[r] & H^1(X,\cB[p^m])\ar[r]\ar[d]^{\cong} & \Sha(B/K)[p^m]\ar[r] & 0\\
0\ar[r] & \cA'(X)/p^m\cA(X)\ar[r] & H^1(X,\cA[p^m])\ar[r] & H^1(X,\cA)[p^m]\ar[r] & 0} \ \ ,
\end{equation}
where the first row in the diagram follows from Proposition~\ref{main:prop},
the second row from Lemma \ref{kummerlemma}, and the vertical isomorphism is the one preceding the diagram. Composition gives a map $$\psi:B(K)/p^mB(K)\to H^1(X,\cA)[p^m]$$ Since there is a natural injection $H^1(X,\cA)\hra H^1(K,A)$ that follows from the Leray spectral sequence, we have the following diagram 
$$\xymatrix{
B(K)/p^mB(K)\ar[r]^{\psi}\ar[dr]^{\vphi} & H^1(X,\cA)\ar@{^{(}->}[d]\\
& H^1(K,A)}$$
In particular, $\ker(\psi)=\ker(\vphi)$ and $\im(\psi)=\im(\vphi)$. Now consider the diagram
$$\xymatrix{
& 0\ar[d]\\
& \Sha(A/K)\ar[d]^{f}\\
B(K)/p^mB(K)\ar[r]^{\psi} & H^1(X,\cA)\ar[d]^{g}\\
& \bigoplus_{v}H^1(\Spec{k_v},\Phi_{A,v})}$$
where the vertical column is the exact sequence (\ref{kernelsha}). If $r_{\ssstyle{B}}>r_{\ssstyle{A}}$, then we have $q_{\ssstyle{A}}<{q_{\ssstyle{B}}}$, where $q_{\ssstyle{A}}=(p^m)^{r_{\ssstyle{A}}}$ and $q_{\ssstyle{B}}=(p^m)^{r_{\ssstyle{B}}}$. Thus we have $$\#\ker(\vphi)\leq{q_{\ssstyle{A}}}<{q_{\ssstyle{B}}}\leq\#B(K)/p^mB(K)$$ This implies that $\im(\psi)\isom\left(B(K)/p^mB(K)\right)/{\left(\ker(\psi)\right)}$ has order at least $\frac{q_{\ssstyle{B}}}{q_{\ssstyle{A}}}$. Now consider the map $$g_{\psi}:\im(\psi)\to\bigoplus_{v}H^1(\Spec{k_v},\Phi_{A,v})$$ where $g_{\psi}$ is the restriction of $g$ to $\im(\psi)$. We find that $\#\im(\psi)$ divides the product $(\#\ker(g_{\psi})\cdot\#\im(g_{\psi}))$. However, $\ker(g_{\psi})\subset\ker(g)=\im(f)$ so that $\#\ker(g_{\psi})$ divides $\#\im(f)=\#\Sha(A/K)$ (assuming $\#\Sha(A/K)$ is finite). On the other hand, $\#\im(g_{\psi})$ divides $\#\bigoplus_{v}H^1(\Spec{k_v},\Phi_{\ssstyle{A},v})=\prod_{v}c_{\ssstyle{A},v}$. We thus conclude that $\#\im(\psi)\geq\frac{q_{\ssstyle{B}}}{q_{\ssstyle{A}}}$ divides the product $(\#\Sha(A/K)\cdot\,\prod_{v}c_{\scriptscriptstyle{A},v})$.

Now suppose that $gcd(p^m,\prod_{v}c_{\scriptscriptstyle{A},v})=1$. By Lemma \ref{lem}, we have an isomorphism $$H^1(X,\cA)[p^m]\isom\Sha(A/K)[p^m]$$ and consequently, we have $$\im(\vphi)=\im(\psi)\subset H^1(X,\cA)[p^m]\isom\Sha(A/K)[p^m]\subset \Sha(A/K)$$ Furthermore, if $r_{\ssstyle{A}}=0$ then $A(K)/p^mA(K)$ is trivial, since $p^m$ is coprime to $\#A(K)_{\tor}$. Since $\ker(\psi)\subset A(K)/p^mA(K)$, this implies that $\ker(\psi)=0$ and we have an injection $$\psi:B(K)/p^mB(K)\hra\Sha(A/K)$$ If we further suppose that $\Sha(B/K)[p^m]$ is trivial, then the third term in the top row and the first term in the bottom row of the commutative diagram (\ref{diagram}) above are both trivial and we find that $\vphi$ is an isomorphism in this case.

\newcommand{\etalchar}[1]{$^{#1}$}

\end{document}